\documentclass[11pt,a4]{amsart} 
\usepackage[pagewise]{lineno}
\usepackage{mathtools}						
\usepackage{enumerate}						
\usepackage{amsthm,amsmath,amsfonts,amsthm,amssymb}
\usepackage[all,cmtip]{xy}
\usepackage{graphicx}
\usepackage{comment}
\usepackage{caption}
\usepackage{marginnote}

\usepackage{booktabs} 
\usepackage{array} 
\usepackage{paralist} 

\usepackage[mathscr]{euscript}						 	%
\usepackage{dsfont}
\usepackage{upgreek}

\def\mfrak{\mathfrak}

\usepackage{textcomp}
\title[]{triple symbols in arithmetic}
\author[D.K.]{Dohyeong Kim}
\author[M.M.]{Masanori Morishita}
\address[D.K.]{Department of Mathematical Sciences and Research Institute of Mathematics, Seoul National University, Gwanak-ro 1, Gwankak-gu, Seoul, South Korea 08826}
\email[D.K.]{dohyeongkim@snu.ac.kr}

\address[M.M.]{Faculty of
Mathematics, Kyushu University, 744, Motooka, Nishi-ku, Fukuoka
819-0395, Japan}
\email[M.M.]{morishita.masanori.259@m.kyushu-u.ac.jp}

\date{Oct 2024} 
\def\Q{\mathbb{Q}}

\def\Z{\mathbb{Z}}

\def\Gal{\operatorname{Gal}}

\def\loc{\operatorname{loc}}
\def\ur{\operatorname{ur}}
\def\inv{\operatorname{inv}}
\def\mod{\operatorname{mod}}
\theoremstyle{theorem}
\newtheorem{theorem}{Theorem}
\newtheorem{lemma}{Lemma}
\newtheorem{proposition}{Proposition}
\newtheorem{corollary}{Corollary}
\theoremstyle{definition}
\newtheorem{definition}{Definition}
\theoremstyle{remark}
\newtheorem{remark}{Remark}

\begin{document}
\begin{abstract}
Triple symbols are arithmetic analogues of the mod $n$ triple linking number in topology, where $n > 1$ is an integer. 
In this paper, we introduce a cohomological formulation of a mod $n$ triple symbol for characters over a number field containing a primitive $n$-th root of unity.
Our definition is motivated by the arithmetic Chern--Simons theory and in this respect it differs from earlier approaches to triple symbols.
We show that our symbol agrees with that of R\'edei when $n=2$ and of Amano--Mizusawa--Morishita when $n=3$.

\end{abstract}
\maketitle

\section{Introduction}
The purpose of this paper is to introduce triple symbols $[ \chi_1, \chi_2, \chi_3]_n$ cohomologically, which takes the values in the group $\mu_n$ of $n$-th roots of unity $(n > 1)$, for certain mod $n$ Kummer characters $\chi_i$'s over a number field $F$ containing $\mu_n$ and to show that our construction provides a new description of known cases of triple symbols for $F = \mathbb{Q}$ and $n = 2$ \cite{Redei1939} and for $F = \mathbb{Q}(\sqrt{-3})$ and $n = 3$ \cite{Amano2018}.
 
The study of triple symbols in arithmetic goes back to the work of R\'{e}dei in 1939 \cite{Redei1939}, which intended to generalize the Legendre symbol and Gauss' genus theory
for quadratic fields. For certain prime numbers $p_1, p_2$ and $p_3$, R\'{e}dei's triple symbol $[p_1, p_2, p_3] \in \{ \pm 1\}$ describes the decomposition law in the dihedral extension $K_{p_1, p_2}$ of degree 8, determined by $p_1$ and $p_2$. Some variants of the R\'{e}dei symbol were also studied in \cite{Froehlich1960, Furuta1980, Suzuki1988} among others. In the late 1990s, the second author interpreted the R\'{e}dei symbol as an arithmetic analogue of the mod 2 triple linking number (Milnor invariant) of a link \cite{Milnor1957, Turaev1976} from the viewpoint of arithmetic topology, and also gave a description in terms of the triple Massey product of the \'{e}tale cohomology of ${\rm Spec}(\mathbb{Z}) \setminus \{ p_1, p_2, p_3 \}$ (cf. \cite{Morishita2004} and \cite[Ch.\,9]{Morishita2024}). However, it remains a problem to extend these constructions over a number field $F$ containing $\mu_n$, because there are cohomological obstructions related to the unit group of $F$. So far only triple cubic residue symbols in $F = \mathbb{Q}(\sqrt{-3})$ could be well defined \cite{Amano2018}. 

In this paper, we take a different approach to the problem to construct triple symbols over a number field $F$ containing $\mu_n$. We start with three mod $n$ Kummer characters $\chi_i : G \rightarrow \mathbb{Z}/n\mathbb{Z}$ ($i = 1, 2, 3$) of the absolute Galois group $G := {\rm Gal}(\overline{F}/F)$ rather than three primes of $F$. Fixing a primitive $n$-th root of unity $\zeta$, we identify $\mathbb{Z}/n\mathbb{Z}$ with $\mu_n$ by $a \; \mbox{mod} \; n \mapsto \zeta^a$. Let $S_i$ be primes where $\chi_i$ is ramified.
We assume that  $S_1, S_2, S_3$ are disjoint, that all $\chi_i$'s are tame characters, and that the local ``linking conditions"
$$
[\operatorname{loc}_v^2 (\chi_i \cup \chi_j)] = 0$$
holds for $i \neq j$ and all non-archimedean places $v$ of $F$. Here, for a place $v$ of $F$, let $\operatorname{loc}_v^i$ denote the localization map of continuous group cochains
$$ \operatorname{loc}_v^i : C^i(G, \mathbb{Z}/n\mathbb{Z}) \rightarrow C^i(G_v, \mathbb{Z}/n\mathbb{Z}), $$
where $G_v$ denotes the local absolute Galois group ${\rm Gal}(\overline{F}_v/F_v)$. Now let us take a finite set $S$ of places of $F$ so that $S$ contains $S_1 \cup S_2 \cup S_3$, any place dividing $n$ and all archimedean places of $F$, and let $G_S$ be the Galois group of the maximal extension of $F$ unramified outside $S$. 
Then, using the assumptions and the vanishing of $H^3(G_S, \mathbb{Z}/n\mathbb{Z})$, we find $\eta \in C^2(G,\mathbb{Z}/n\mathbb{Z})$ and $\eta_v \in C^2(G_v,\mathbb{Z}/n\mathbb{Z})$ such that 
$$ d\eta = \chi_1 \cup \chi_2 \cup \chi_3, \;\; d\eta_v = \operatorname{loc}_v^2 (\chi_1 \cup \chi_2 \cup \chi_3),$$
where $d$ denotes the coboundary map. 
Here $\eta_v$ is required to be unramified, a condition to be introduced in the section 2. The key idea to detect the triple symbol is to look at the difference of global cochain $\eta$ and local unramified cochain $\eta_v$ over $v \in S$. Thus, we set
$$
\frak{d}(\chi_1, \chi_2, \chi_3) := \sum_{v \in S} \operatorname{inv}_v(\operatorname{loc}_v^2(\eta) - \eta_v) \in \mathbb{Z}/n\mathbb{Z},
$$
where $\operatorname{inv}_v$ is the invariant map $H^2(G_v, \mathbb{Z}/n\mathbb{Z}) \stackrel{\sim}{\rightarrow} \mathbb{Z}/n\mathbb{Z}$ of local class field theory for a non-archimedean place $v$ or a real place $v$ with $n=2$, and $\operatorname{loc}_v = 0$ for other cases. 
Here, as we will prove in Prop.\,\ref{prop:3}, among other things, that the sum $\frak{d}(\chi_1, \chi_2, \chi_3)$ is independent of a choice of $S$. We then define the triple symbol by
$$ [\chi_1, \chi_2, \chi_3]_n := \zeta^{\frak{d}(\chi_1, \chi_2, \chi_3)}.$$
The idea to consider $\frak{d}(\chi_1,\chi_2,\chi_3)$ is motivated by the first author's computation of arithmetic Chern-Simons invariants \cite{Chung2020}.

As we will show in Cor.\,\ref{cor:alt}, the triple symbol turns out to be alternating.
In particular, when $n=3$, our result confirms a conjecture by Shiraishi~\cite[A.0.5]{Shiraishi2021}.

We briefly discuss the practical aspects about computation of $\mfrak d(\chi_1,\chi_2,\chi_3)$ by means of evaluating the sum above.
Since its number of terms is the cardinality of $S$, it is desirable to work with a small $S$.
On the other hand, $S$ cannot be too small otherwise the necessary cochain $\eta$ will not exist.
Theoretically, it is convenient to require that $S$ contains all $S_i$'s, all places dividing $n$, and all archimedean places, because the combination of them will gaurantee the existence of $\eta$ for any triple of $\chi_i$'s.
Practically, it is computationally effecient to choose a smaller $S$, if possible, for the particular triple of $\chi_i$'s in consideration.
For this reason, we sometimes relax the condition on $S$, notably in Prop.\,\ref{prop:3} and \textsection\,\ref{section:redei}.

The R\'{e}dei's symbol is then recovered from  ours as follows. Choose $\chi_i$ to be the quadratic Kummer character associated to $p_i$. The conditions imposed on $p_1, p_2, p_3$ to define R\'{e}dei's symbol implies that our conditions on $\chi_1, \chi_2, \chi_3$ holds true with $n=2$ and $S = \{ p_1, p_2, p_3, 2, \infty\}$. Then we show $[p_1,p_2,p_3] = [\chi_1,\chi_2,\chi_3]_2$. Similarly, we recover the triple cubic symbol in \cite{Amano2018} as a special case of our symbol. These are shown in \textsection\,\ref{section:redei},\,\ref{section:eisenstein}. Moreover, for these cases, we can show that $\frak{d}(\chi_1,\chi_2, \chi_3)$ gives another description of the Massey product $\langle \chi_1, \chi_2,\chi_3 \rangle$ in \textsection\,\ref{section:massey}.
\\
{\it Acknowledgement}. This work was started during the second author's stay at Seoul National University in October of 2023. The second author is thankful to the first author and SNU for the support and hospitality. 
The first author is supported by the NRF\footnote{National Research Foundation of Korea, No. 2020R1C1C1A01006819}.
We thank Yan Yau Cheng for pointing out an error in an earlier version of the paper.
\\

\section{triple symbols}\label{section:2}
Let $n \ge 2$ be an integer.
Put
\begin{align}\label{def:dn}
D_n := \left\{
\begin{pmatrix}
1&a&b
\\
0&1&c
\\
0&0&1
\end{pmatrix}
\colon a,b,c \in \mathbb Z/n\mathbb Z
\right\}.
\end{align}
When $n$ is unambiguous we simply write $D=D_n$.
This group sits in the short exact sequence of groups
\begin{align}\label{eq:ses}
1 \to Z \to D \xrightarrow{} \overline D \to 1
\end{align}
where $Z \subset D$ denotes the center and $\overline D := D/Z$.
Concretely speaking, $Z$ is the subgroup given by $a=c=0$.
We will identify $Z \simeq \mathbb Z/n\mathbb Z$ via the projection onto the $b$-component.
\par
A standard argument~\cite[\textsection\,I.5.7 Proposition 43]{Serre1994} shows that for any profinite group $G$ acting trivially on $D$, taking continuous cohomology sets associated to (2) yields an exact sequence of pointed sets
\begin{align}\label{eq:boundary}
H^1(G,D) \to H^1(G, \overline D) \xrightarrow{\delta} H^2(G,\mathbb Z/n\mathbb Z).
\end{align}

\begin{lemma}\label{lemma}
A homomorphism $\varphi \colon G \to \overline D$ can be lifted to a continuous homomorphism $G \to D$ if and only if $[\delta\varphi] = 0 $.
\end{lemma}
\begin{proof}
Immediate from \eqref{eq:boundary}.
\end{proof}
We describe the class $\delta \varphi$ in terms of the cup product.
We identify $\overline D \simeq \mathbb Z/n\mathbb Z \times \mathbb Z/n\mathbb Z$ by sending a coset containing 
$\begin{psmallmatrix}
1&a&b
\\
0&1&c
\\
0&0&1
\end{psmallmatrix}$
to $(a,c)$. 
\begin{lemma}
Suppose that $\varphi \colon G \to \overline D$ is given by $\varphi(g) = \left(\chi_1(g),\chi_2(g)\right)$ for some homomorphisms $\chi_1,\chi_2 \colon G \to \mathbb Z/n\mathbb Z$.
Then, $\delta\varphi = \chi_1 \cup \chi_2$.
\end{lemma}
\begin{proof}
For $g \in G$, let $\tilde{\varphi} (g)$ be a lift of $\varphi(g)$ to $D$. Namely $\tilde{\varphi}(g)=s \circ \varphi(g)$ for some section $s \colon \overline D \to D$ of the projection map in \eqref{eq:ses}.
Then, there is a $1$-cochain $\alpha \in C^1(G,\mathbb Z/n\mathbb Z)$ such that
$$
\tilde\varphi(g) = 
\begin{pmatrix}
1&\chi_1(g)&\alpha(g)
\\
0&1&\chi_2(g)
\\
0&0&1
\end{pmatrix}.
$$
The class $\delta [\varphi]$ is represented by the 2-cocycle $(g_1,g_2) \mapsto \tilde\varphi(g_1)\tilde\varphi(g_2)\tilde\varphi(g_1g_2)^{-1}$.
By straightforward computation, we obtain $\tilde\varphi(g_1)\tilde\varphi(g_2)\tilde\varphi(g_1g_2)^{-1} = \chi_1\cup\chi_2(g_1,g_2) + d\alpha(g_1,g_2)$.
This proves the assertion of the lemma.
\end{proof}
Let $F$ be a number field and $n$ be an integer greater than one such that $F$ contains a primitive $n$-th root of unity.
Denote by $\mu_n \subset F$ the subset of all $n$-th roots of unity.
We fix a generator $\zeta \in \mu_n$ and identify $\mu_n \simeq \mathbb Z/n\mathbb Z$ by sending $\zeta$ to $1$.
Now we take $G=\Gal(\overline F / F)$ for an algebraic closure $\overline F$ of $F$. 
For any set $S$ of places of $F$, let $G_S$ be the quotient of $G$ corresponding to the maximal extension of $F$ unramified outside $S$.
Consider characters
\begin{align*}
\chi_i \colon G\to \mathbb Z/n\mathbb Z
\end{align*}
for $i=1,2,3$. Since $F$ contains $\mu_n$, $\chi_i$ is a Kummer character associated to some $\alpha_i \in F^{\times}$, namely, $\zeta^{\chi_i(g)} = g(\sqrt[n]{\alpha_i})/\sqrt[n]{\alpha_i}$ for $g \in G$. 
Denote by $S_i$ the set of places where $\chi_i$ is ramified.
\par
\begin{proposition}\label{prop:1}
Let $S$ be a finite set of places of $F$ such that $S_i \subset S$ for $i=1,2,3$ and that $S$ contains both all places $v$ dividing $n$ and all archimedean places. 
There is a $2$-cochain $\eta \in C^2(G,\mathbb Z/n \mathbb Z)$ such that
$$
d\eta = \chi_1 \cup \chi_2 \cup \chi_3
$$
and that $\eta$ is in the image of the inflation map $C^2(G_S,\mathbb Z/n\mathbb Z) \to C^2(G,\mathbb Z/n\mathbb Z)$.
\end{proposition}
\begin{proof}
Observe that $\chi_1\cup\chi_2\cup\chi_3$ represents a cocycle in $C^3(G_S,\mathbb Z/n\mathbb Z)$.
The assertion of the proposition would follow from $H^3(G_S,\mathbb Z/n\mathbb Z)=0$.
Indeed, this is a standard result \cite{Neukirch2008} about cohomological dimension of $G_S$.
\end{proof}
\par
For a place $v$ of $F$, choose an algebraic closure $\overline F_v/F_v$ and put $G_v=\Gal(\overline F_v/F_v)$. 
Fix an embeeding $\overline F \to \overline F_v$ and regard $G_v$ as a subgroup of $G$.
In particular, restriction induces a map between cochains which we denote as
$$
\loc_v^i \colon C^i(G,\mathbb Z/n\mathbb Z) \to C^i(G_v,\mathbb Z/n\mathbb Z)
$$
for each $i \ge 0$.
By a local trivialization of a cocycle $\varphi \in C^i(G,\mathbb Z/n\mathbb Z)$, we mean a cochain
$$
\eta_v \in C^{i-1}(G_v,\mathbb Z/n\mathbb Z)
$$
such that
$$
d\eta_v = \loc_v^i(\varphi).
$$
When $v$ is non-archimedean, we put $G_v^{\ur}:= \Gal(F_v^{\ur}/F_v)$ where $F_v^{\ur}\subset \overline F_v$ is the maximal unramified subextension of $F_v$. 
We say a local trivialization $\eta_v$ is unramified if $\eta_v$ is in the image of the map
$$
C^{i-1}(G_v^{\ur},\Z/n\Z) \to C^{i-1}(G_v,\Z/n\Z)
$$
induced by inflation.
If $v$ is archimedean, by convention, a local trivialization is unramified if it is a constant function.
\par
On the other hand, we say a cochain $\eta \in C^{i-1}(G,\Z/n\Z)$ is a global trivialization unramified outside $S$ if $d\eta = \varphi$ and $\eta$ is in the image of the map
$$
C^{i-1}(G_S,\Z/n\Z) \to C^{i-1}(G,\Z/n\Z)
$$
induced by the inflation along the natural surjection $G \to G_S$.
\begin{proposition}\label{prop:2}
Let $v$ be a non-archimdean place of $F$. Assume that $S_1$, $S_2$, $S_3$ are disjoint and that each $\chi_i$ is tame; the wild inertia group inside $G_v$ is contained in the kernel of $\chi_i \circ \loc_v$.
Further assume that
$$
[\loc_v(\chi_i \cup \chi_j)]=0
$$
for all $i\not = j$. 
Then, $\loc_v^3 \left(\chi_1 \cup \chi_2 \cup \chi_3\right)$ is trivialized by an unramified cochain.
\end{proposition}
\begin{proof}
Put $\psi_i = \loc_v \circ \chi_i$. 
We first prove the assertion when $v$ is non-archimedean.
We do not lose generality by assuming that $n$ is a power of a prime $\ell$, because $C^i(-,A\oplus B)\simeq C^i(-,A)\oplus C^i(-,B)$. 
We claim that $\psi_1 \cup \psi_2 \cup \psi_3=0$ as a cochain, or all of $\psi_i$ are unramified.
It would imply the assertion of the proposition, since $H^3(G_v^{\ur},\Z/n\Z)=0$. 
To prove the claim, let us assume that one of $\psi_i$ is ramified, say $\psi_1$.
Then, by the assumptions, for $i\in\{2,3\}$, $\psi_i$ is unramified.
We claim that $\psi_1 \cup \psi_2=0$ as a cochain.
For $i\in\{1,2\}$, let $d_i$ be the order of $\psi_i$.
Then, $\psi_1$ is tamely ramified so it is of the form $\psi_1  = \psi_{\operatorname{tm}}^{n/d_1}$ for a tamely ramified character $\psi_{\operatorname{tm}}$ of order $n$. 
On the other hand, $\psi_2 = \psi_{\ur}^{n/d_2}$ for an unramified character $\psi_{\ur}$ of order $n$.
Then, the condition $[\psi_1 \cup \psi_2]=0$ is equivalent to $d_1d_2 \mid n$, by local class field theory. 
Since the cup product is bilinear at the level of cochains we have a cochain-level equality $\psi_1 \cup \psi_2=\frac{n^2}{d_1d_2}\left(\psi_{\operatorname{tm}}\cup\psi_{\ur}\right)$. 
Now $d_1d_2\mid n$ implies the aformentioned claim.
From this, we conclude that $\psi_1\cup\psi_2\cup\psi_3=0$ as a cochain.

\par
Now consider the case when $v$ is archimedean.
The disjointness of $S_i$'s implies that at most one of $\psi_v$'s is non-trivial. 
At least two of them are trivial, so we again conclude $\psi_1\cup\psi_2\cup\psi_3=0$ as a cochain.
\end{proof}
Recall that we assumed $\mu_n \subset F$ and fixed an isomorphism $\mu_n \simeq \mathbb Z/n\Z$. The local class field theory provides an isomorphism
$$
\inv_v \colon H^2(G_v,\Z/n\Z) \xrightarrow{\sim} \Z/n\Z
$$
for a non-archimedean place $v$.
If $v$ is archimedean we have $H^2(G_v,\Z/n\Z)=0$ unless $v$ is real and $n$ is even, in which case $\inv_v \colon H^2(G_v,\Z/n\Z) \simeq \Z/2\Z$.

In the next proposition, we temporarily relax the condition on $S$ while in Prop.\ref{prop:1} the set $S$ was required, among other things, to contain all places dividing $n$.
The advantage will be justified in Rem\,\ref{remark:2}.

\begin{proposition}\label{prop:3}
Let $S$ be any set of places containing all archimedean places and $S_i$ for all $i$. 
Suppose that $\chi_1 \cup \chi_2 \cup \chi_3$ admits a global trivialization unramified outside $S$, say $\eta$, and that it admits a local unramified trivialization $\eta_v$ for every place $v$ of $F$.
The sum
\begin{align*}
\frak{d}(\chi_1,\chi_2,\chi_3)=\sum_{v\in S} 
\inv_v \left( \loc_v^2 \eta - \eta_v\right) \in \Z/n \Z
\end{align*}
 is independent of the choices of trivializations.
Also, $\frak{d}(\chi_1,\chi_2,\chi_3)$ is independent of $S$.
\end{proposition}
\begin{proof}
For two global trivializations $\eta$ and $\eta'$, $\eta-\eta'$ is a class in $H^2(G,\Z/n\Z)$. 
By the global reciprocity, the sum $\frak{d}(\chi_1,\chi_2,\chi_3)$ remains the same.
On the other hand, for a place $v$ and two unramified trivializations $\eta_v$ and $\eta_v'$, the difference $\eta_v-\eta_v'$ is unramified.
In particular, the difference is the image of a class of $H^2(G_v^{\ur},\Z/n\Z)$. 
The group $H^2(G_v^{\ur},\Z/n\Z)$ is trivial because $G_v^{\ur}\simeq \hat{\Z}$ has cohomological dimension one.
So $\frak{d}(\chi_1,\chi_2,\chi_3)$  is independent of the choices of $\eta$ and $\eta_v$'s.

Now we prove the independence of $S$.
It suffices to show that $\frak{d}(\chi_1,\chi_2,\chi_3)$  remains the same if we replace $S$ with a larger set $S'$.
Suppose $\chi_1 \cup \chi_2 \cup \chi_3$ admits a global trivialization unramified outside of $S$, say $\eta$.
If $S' \supset S$, then $\eta$ is a global trivialization unramified outside $S'$ because the surjection $G \twoheadrightarrow G_S$ factors through $G \twoheadrightarrow G_{S'}(F)$. 
We claim that for $v \in S'-S$, we have $\inv_v\left( \loc_v^2 \eta - \eta_v\right)=0$. 
Indeed, $\loc_v^2\eta$ is in the image of $C^2(G_v^{\ur},\Z/n\Z) \to G^2(G_v,\Z/n\Z)$, because $\eta$ is in the image of
the inflation map $C^2(G_S,\Z/n\Z) \to C^2(G_{S'}(F),\Z/n\Z)$.
It follows that $\loc_v^2 \eta - \eta_v$ is unramified. 
From $H^2(G^{\ur}_v,\Z/n\Z)=0$, we conclude $\inv_v \left( \loc_v^2 \eta - \eta_v\right)=0$.
We have shown that $\frak{d}(\chi_1,\chi_2,\chi_3)$  is independent of $S$.
\end{proof}
\begin{remark}
Its proof is similar to that for decomposition formula \cite[\textsection\,4]{Chung2020} in the arithmetic Dijkgraaf-Witten theory.
\end{remark}
\begin{remark}\label{remark:2}
In Prop.\,\ref{prop:1} we require that $S$ contains all places dividing $n$ in order to ensure the existence of a global trivialization $\eta$.
However, there might exist a global trivialization $\eta'$ unramified outside a proper subset $S'$ of $S$.
On the other hand, in Prop.\,\ref{prop:3} the condition on $S$ is partially relaxed in order to allow such $S'$.
When $\eta'$ is available, we use it instead of $\eta$ to obtain a shorter sum.
\end{remark}
Recall that we have fixed a primitive $n$-th root of unity $\zeta$ which maps to $1$ under the isomorphism $\mu_n \simeq \Z/n\Z$.
\begin{definition}\label{def:triple}
Suppose that $\chi_1,\chi_2$ and $\chi_3$ satisfy the assumptions in Prop.\,\ref{prop:3}. Define 
\begin{align}\label{eq:def2}
[\chi_1,\chi_2,\chi_3]_n := \zeta^{\frak{d}(\chi_1,\chi_2,\chi_3)}.
\end{align}
\end{definition}\begin{theorem}\label{thm:1}
Suppose that $S_i$'s are disjoint and that $\loc _v^1 \chi_i$ is tame for all $i$ and all $v$. 
If $[\loc_v^2 \left(\chi_i \cup  \chi_j\right)]=0$ for all non-arhimedean $v$ and all $1\le i<j\le 3$, then
\eqref{eq:def2} is defined.
\end{theorem}
\begin{proof}
By Proposition 1, global trivialization $\eta$ exists if $S$ contains both all places dividing $n$ and all archimedean places.
On the other hand, by Proposition 2, an unramified trivialization exists at every non-arhimedean $v$.
For an archimdean place $v$, the disjointness of $S_i$'s imply that at least one of $\chi_i$'s is trivial at $v$. 
It follows that $\loc_v^3(\chi_1 \cup \chi_2 \cup \chi_3)=0$ so it is trivialized by an unramified cochain.
The remaining independence follows from Proposition 3.
\end{proof}

\begin{corollary}\label{cor:alt}
For any permuation $\sigma$ on the set $\{1,2,3\}$ with sign $|\sigma|$,
we have
\begin{align}\label{eq:alt}
[\chi_1,\chi_2,\chi_3]_n=[\chi_{\sigma(1)},\chi_{\sigma(2)},\chi_{\sigma(3)}]_n^{|\sigma|}
\end{align}
whenever the triple symbol is defined.
\end{corollary}
\begin{proof}
Since the cup product is alternating, we have
$$
 \chi_{1} \cup \chi_{2} \cup \chi_{3}
=
|\sigma|
 \left(
\chi_{\sigma(1)} \cup \chi_{\sigma(2)} \cup \chi_{\sigma(3)}
\right).
$$
The conditions in Thm.\,\ref{thm:1} are indifferent to the ordering of the characters, so the left hand side of \eqref{eq:alt} is defined if and only if the right hand side is defined.
If $\eta$ is a global trivialization for $\chi_1 \cup \chi_2 \cup \chi_3$, then $|\sigma|\eta$ is one for $\chi_{\sigma(1)} \cup \chi_{\sigma(2)} \cup \chi_{\sigma(3)}$. 
Similarly, if $\eta_v$ is a local unramified trivialization for $\chi_1 \cup \chi_2\cup \chi_3$, then $|\sigma|\eta_v$ is one for $\chi_{\sigma(1)} \cup \chi_{\sigma(2)} \cup \chi_{\sigma(3)}$. 
From this, one concludes \eqref{eq:alt}.
\end{proof}

\section{R\'edei symbols}\label{section:redei}
In \cite{Redei1939}, R\'edei introduced a triple symbol $[d_1,d_2,d_3]$ for non-zero integers $d_i \in \mathbb Z$.
They are defined only when certain conditions are satisfied among $d_i$'s. 
Later, variants are studied in \cite{Froehlich1960} and \cite{Furuta1980}.

Our aim in this section is to recall the case when $d_i=p_i$ is a prime, and show that R\'edei's symbol agrees with ours. 
Let $p_1,p_2$ and $p_3$ be distinct prime numbers such that $p_i \equiv 1\, \mod 4$ ($i = 1, 2, 3$).
When
$$
\left(\frac{p_i}{p_j}\right)=1
$$
for all $1 \le i < j \le 3$, R\'edei defined his triple symbol
$$
[p_1,p_2,p_3]\in \{\pm 1\}
$$
which we reproduce here.
\begin{definition}
Let $p_1, p_2$ be distinct prime numbers such that $p_1 \equiv p_2 \equiv 1\,(\mod 4)$.
A R\'edei extension associated to $p_1$ and $p_2$ is a Galois extension $K/\mathbb Q$ such that
\begin{enumerate}
\item $K/\mathbb{Q}$ is unramified outside $p_1, p_2$ and $\infty$,
\item the Galois group ${\rm Gal}(K/\mathbb{Q})$ is isomophic to $D_2$ in (1), the dihedral group of order 8, and
\item the ramification indices of $p_1$ and $ p_2$ are exactly 2.
\end{enumerate}
\end{definition}
\begin{proposition}[R\'edei and  Amano]\label{thm:2}
Suppose that $\left(\frac{p_1}{p_2} \right)=1$. 
Then, a R\'edei extension exists and is unique. 
It is given by 
\begin{align}
K_{p_1,p_2}:=\Q(\sqrt p_1, \sqrt p_2, \sqrt \beta)
\end{align}
where $\beta = x + y \sqrt p$ satisfies the following conditions:
\begin{enumerate}
\item $2 \mid y$ and $x-y \equiv 1\,\mod 4$,
\item $x^2-p_1 y^2- p_2z^2-0$ for some $z \in \mathbb Z$, and 
\item $x,y,z$ are relatively prime.
\end{enumerate}
\end{proposition}
\begin{proof}
See Theorem 1.2 and Theorem 2.1 of \cite{Amano2014}.
\end{proof}
If $p_3$ is a prime not dividing $p_1 p_2$ and satisfies $\left(\frac{p_1}{p_3} \right)= \left(\frac{p_2}{p_3} \right)=1$, then $p_3$ is totally decomposed in $\mathbb Q(\sqrt p_1 , \sqrt p_2)$.
Thus, the degree of $r$ in $K_{p_1,p_2}$, say $f(K_{p_1,p_2},p_3)$, is either one or two.
R\'edei's triple symbol is defined as
\begin{align}
[p_1,p_2,p_3] :=
\begin{cases}
1 &\text{if  $f(K_{p_1,p_2},p_3)=1$},
\\
-1 &\text{if  $f(K_{p_1,p_2},p_3)=2$}.
\end{cases}
\end{align}

We want to recover the R\'{e}dei symbol $[p_1,p_2,p_3]$ from Definition 1.
\begin{theorem}\label{thm:compare-redei}
Let $p_1,p_2$ and $p_3$ be prime numbers such that $p_i \equiv 1\,\mod 4$ ($i = 1,2,3$) and that $\left(\frac{p_i}{p_j}\right)=1$ ($1\le i<j\le 3)$.
Denote by $\chi_i : G \rightarrow \mathbb{Z}/2\mathbb{Z}$ the character defined  by $(-1)^{\chi_i(g)} = g(\sqrt{p_i})/\sqrt{p_i}$ for $g \in G$. 
Then, $[\chi_1,\chi_2,\chi_3]_2$ is defined and satisfies 
$$[\chi_1,\chi_2,\chi_3]_2 = [p_1,p_2,p_3].$$
\end{theorem}
\begin{proof}
We first note that $[\chi_1,\chi_2,\chi_3]_2$ is defined by Theorem 1.
So it remains to compare it to R\'edei's symbol.
Let $K_{p_1,p_2}$ be the unique R\'edei extension given by (6) in  Proposition 4.
It gives a homomorphism $G \to D_2$.
This lifts the homomorphism $G \to \Z/2\Z \times \Z/2\Z$ given by $\chi_1 \times \chi_2$. 
By Lemma 1, this gives a trivialzation $\epsilon \in C^1(G,\Z/2\Z)$ for $\chi_1 \cup \chi_2$ unramified outside $\{p_1,p_2,\infty\}$.
Put
$$
\eta = \epsilon \cup \chi_3
$$
then $\eta$ is a trivialization of $\chi_1 \cup \chi_2 \cup \chi_3 $ unramified outside $S=\{p_1,p_2,p_3,\infty\}$.

Table~\ref{table:redei} shows the upshot of local computations that are needed to determine $[\chi_1,\chi_2,\chi_3]_2$.
In view of Remark.\,\ref{remark:2}, we note that there are no contributions from the place $2$, becauce $\eta$ is unramified at $2$.
The first row indicates that $K_{p_1,p_2}$ is ramified at $p_1$ and $p_2$ and possibly at $\infty$.
The second row indicates that $\chi_3$ is ramified exactly at $p_3$ and trivial at either $p_1$ or $p_2$.
The third row indicates that the local unramified trivializations are taken to be zero whenever $\loc_v^3(\chi_1\cup\chi_2\cup\chi_3)$ is identically zero as a cochain.
The fourth row indicates that the constraints in the first three columns forces $\inv_v(\loc_v^2(\eta)-\eta_v)$ to be zero except for $v=p_3$.
So we obtain $[\chi_1,\chi_2,\chi_3]_2=(-1)^m$.

\begin{table}[htb]
\def\arraystretch{1.2}
\begin{tabular}{|c||c|c|c|c|c|}
\hline
$v$		&	$p_1$	or  	$p_2$ 	& 	$p_3$		&		$\infty$
\\
\hline 
$\operatorname{loc}_v^2\left(\eta\right)$
				& ramified 		&unramified 	 & possibly ramified
\\
\hline
$\loc_v^1 \circ\chi_3$ 			&$0$			&ramified	 &0
\\
\hline
$\eta_v$ 		&	$0$ 			& 0	&0 
\\
\hline
$\inv_v\left(\operatorname{loc}_v^2\left(\eta\right) - \eta_v\right)$
				&	$0$			&	$m$				&$0$
\\
\hline
\end{tabular}
\caption{local terms in the case $F=\mathbb Q$ and $n=2$}
\label{table:redei}
\end{table}
\par
To finish the proof, it suffices to express $m$ in terms of the decomposition of $p_3$ in $K_{p_1,p_2}$ and compare the result to the definition (7) of $[p_1,p_2,p_3]$.
Since $\eta = \epsilon \cup \chi_3$ and $\eta_{p_3}=0$, we need to evaluate $\inv_{p_3}((\loc_{p_3}^1 \epsilon) \cup (\loc_{p_3}^1 \circ\chi_3))$.
Put $\psi_3=\loc_{p_3}^1 \epsilon$. 
Then $\psi_3$ is a character of $G_{p_3}^{\ur}$ whose order is at most two.
Moreover, it is trivial if and only if $p_3$ splits completely in $K_{p_1,p_2}$.
Since $\loc_{p_3}^1 \circ \chi_3$ is the Kummer character associated to $p_3$, the local class field theory implies $\inv_{p_3}(\psi_3\cup(\loc_{p_3}^1 \circ\chi_3))=\psi_3(p_3)$. 
So we conclude $m=0$ if $p_3$ splits completely in $K_{p_1,p_2}$, and $m=1$ if $p_3$ has degree two in $K_{p_1,p_2}$, as desired.
\end{proof}
Cor.1 yields the following reciprocity law of R\'{e}dei cohomologically.
\begin{corollary}\label{cor:2}
For a permulation $\sigma$ on $\{1,2,3\}$, we have
$$[p_1,p_2,p_3]=[p_{\sigma(1)},p_{\sigma(2)},p_{\sigma(3)}].$$
\end{corollary}
\noindent

\section{Triple cubic residue symbols in the Eisenstein field}\label{section:eisenstein}
In  this section, let $F=\Q(\zeta_3)$ where $\zeta_3$ is a fixed primitive third root of unity. We recall the cubic triple symbol  introduced in \cite{Amano2018}.
 
Let  $\mathfrak p_1, \mathfrak p_2$ and $\mathfrak p_3$ be distinct maximal ideals of $\mathbb{Z}[\zeta_3]$ such that ${\rm N}\frak{p}_i  \equiv 1$ mod 9 ($i=1,2,3$). Note that ${\rm N}\frak{p}_i \equiv 1$ mod 9 if and only if there is a prime element  $\pi_i \in \mathbb{Z}[\zeta_3]$ such that $\frak{p}_i = (\pi_i), \pi_i \equiv 1$ mod $(3\sqrt{-3})$. In the following, we fix such prime elements $\pi_i$'s and $\left(\frac{\cdot}{\cdot}\right)_3$ denotes the cubic residue symbol in $F$. When
$$
\left(\frac{\pi_i}{\pi_j}\right)_3=1
$$
for all $1 \le i < j \le 3$, the triple cubic residue symbol 
$$
[ \frak{p}_1, \frak{p}_2, \frak{p}_3]_3 \in \mu_3
$$
is defined as follows. 
\begin{definition} Let $\frak{p}_1, \frak{p}_2$ be distinct maximal ideals of $\mathbb{Z}[\zeta_3]$ such that ${\rm N}\frak{p}_1 \equiv {\rm N}\frak{p}_2 \equiv 1$ (mod 9). 
A R\'edei type extension $K$ of $F = \mathbb{Q}(\zeta_3)$ associated to $\frak{p}_1, \frak{p}_2$ is a Galois extension such that
\begin{enumerate}
\item it is unramified outside $\mathfrak p_1$ and $\mathfrak p_2$,
\item its Galois group ${\rm Gal}(K/F)$ is isomophic to $D_3$, and
\item its ramification indices in $ F_{\mathfrak p_1,\mathfrak p_2}$ of $\mathfrak p_1$ and $\mathfrak p_2$ are exactly 3.
\end{enumerate}
\end{definition}
\begin{theorem}[Amano, Mizusawa and Morishita]\label{thm:4} Notations being as in Definition 3, a R\'{e}dei type extension of $F$ assocciated to $\frak{p}_1, \frak{p}_2$ is unique if it exists. 
Let $\frak{p}_1 = (\pi_1), \frak{p}_2 = (\pi_2)$ with $\pi_1 \equiv \pi_2 \equiv 1$ $\mbox{mod}\, (3\sqrt{-3})$. Suppose that $\left( \frac{\pi_1}{\pi_2} \right)_3=1$. 
Then, a R\'edei extension exists if $\pi_1$ and $\pi_2$ are rational primes. 
It is given by the following form
\begin{align}
F_{\frak{p}_1,\frak{p}_2}:= F(\sqrt[3]{\pi_1}, \sqrt[3]{\pi_2}, \sqrt[3]{\theta})
\end{align}
where $\theta$ is an element of $F(\sqrt[3]{\pi_1})$ which is given explicitly. 
\end{theorem}
\begin{proof}
See Theorem 4.1, Corollary 5.9, Theorem 5.7 and Theorem 5.11 of \cite{Amano2018}
\end{proof}

Let $F_{\frak{p}_1,\frak{p}_2}$ be the R\'{e}dei type extension given by (8) in Proposition 5. Let $\frak{p}_3 = (\pi_3)$ be a maximal ideal of $\mathbb{Z}[\zeta_3]$ such that it is prime to $\frak{p}_1 \frak{p}_2$ and ${\rm N}\frak{p}_3 \equiv 1$ mod 9, equivalently, $\pi_3 \equiv 1$ (mod $(3\sqrt{-3})$). We assume that $\left(\frac{\pi_1}{\pi_3} \right)_3= \left(\frac{\pi_2}{\pi_3} \right)=1$. Let $\frak{P}_3$ be a place of $F_{\frak{p}_1,\frak{p}_2}$ lying over $\frak{p}_3$.  Since $\frak{P}_3$ is unramified in $F_{\frak{p}_1, \frak{p}_2}/F$, the Artin symbol $\left(\frac{F_{\mathfrak p_1,\mathfrak p_2}/F}{\frak{P}_3}\right) \in {\rm Gal}(F_{\frak{p}_1,\frak{p}_2}/F)$ is defined. Since $\frak{p}_3$ is totally decomposed in $F(\sqrt[3]{\pi_1} , \sqrt[3]{\pi_2})$ by the assumption, $\left(\frac{F_{\mathfrak p_1,\mathfrak p_2}/F}{\frak{P}_3}\right)$ is independent of a choice of $\frak{P}_3$ over $\frak{p}_3$ and so we denote it by $\left(\frac{F_{\mathfrak p_1,\mathfrak p_2}/F}{\frak{p}_3}\right)$.
Then, by Theorem 6.3 of \cite{Amano2018}, the triple cubic residue symbol is given by 
\begin{align}\label{eq:eis}
[\mathfrak p_1,\mathfrak p_2,\mathfrak p_3]_3 =
\left(
	\left(\frac{F_{\mathfrak p_1,\mathfrak p_2}/F}{\mathfrak p_3}\right)
	 \sqrt[3]{\theta}
\right)/\sqrt[3]{\theta}.
\end{align}
For our purpose, we adopt the right hand side of (9) as definition of $[\mathfrak p_1,\mathfrak p_2,\mathfrak p_3]_3$.
\begin{theorem}
Let $\frak{p}_1 = (\pi_1), \frak{p}_2 = (\pi_2)$ and $\frak{p}_3 = (\pi_3)$ be distinct maximal ideals of $\mathbb{Z}[\zeta_3]$ such that ${\rm N}\frak{p}_i  \equiv 1\,\mod 9$ $(i = 1,2,3)$ and that $\left(\frac{\pi_i}{\pi_j}\right)_3=1$ $(1\le i<j\le 3)$.
Denote by $\chi_i : G \rightarrow \mathbb{Z}/3\mathbb{Z}$ the character defined by $\zeta_3^{\chi_i(g)} = g(\sqrt[3]{\pi_i})/\sqrt[3]{\pi_i}$ for $g \in G$. 
Then, $[\chi_1,\chi_2,\chi_3]_3$ is defined and satisfies 
$$[\chi_1,\chi_2,\chi_3]_3 = [\frak{p}_1, \frak{p}_2, \frak{p}_3]_3^{-1}.$$
\end{theorem}
\begin{proof} Let $L := F(\sqrt[3]{\pi_1},\sqrt[3]{\pi_2})$ and let $\mathfrak p_3'$ is a place of $L$ lying below $\mathfrak P_3$ and let $\psi_a : {\rm Gal}((F_{\frak{p}_1, \frak{p}_2})_{\frak{P}_3}/ L_{\mathfrak p_3'}) \rightarrow \mathbb{Z}/3\mathbb{Z}$ be the Kummer character associated to $a \in L_{\mathfrak p_3'}^{\times}$, namely, $\zeta_3^{\psi_a(g)} = g(\sqrt[3]{a})/\sqrt[3]{a}$ for $g \in {\rm Gal}((F_{\frak{p}_1, \frak{p}_2})_{\frak{P}_3}/L_{\mathfrak p_3'})$. Let ${\rm inv}_{\frak{p}_3'} : H^2({\rm Gal}((F_{\frak{p}_1, \frak{p}_2})_{\frak{P}_3}/L_{\mathfrak p_3'}), \mathbb{Z}/3\mathbb{Z}) \rightarrow \mathbb{Z}/3\mathbb{Z}$
 be the invariant map of local class field theory given under the identification $\mathbb{Z}/3\mathbb{Z} \simeq \mu_3$ by $1 \mapsto \zeta$. By the known relation between the norm residue symbol and the cup product (Theorem 8.12 of \cite{Koch2002}), we have 
\begin{align}\label{eq:13}
\left(
	\left(\frac{F_{\mathfrak p_1,\mathfrak p_2}/F}{\mathfrak p_3}\right)
	 \sqrt[3]{\theta}
\right)/\sqrt[3]{\theta}
= \zeta^{-\operatorname{inv}_{\mathfrak p_3'} \left(\psi_{\theta} \cup \psi_{\pi_3}\right)}.
\end{align}
Note that $\mathfrak p'_3$ is degree one prime unramified over $\mathfrak p_3$ so $\left(F_{\mathfrak p_1,\mathfrak p_2}\right)_{\mathfrak p_3'} \simeq F_{\mathfrak p_3}$. 
From this we have
$$
 \operatorname{inv}_{\mathfrak p_3'} \left(\psi_{\theta} \cup \psi_{\pi_3}\right)
=
 \operatorname{inv}_{\mathfrak p_3} \left(\psi_{\theta} \cup \psi_{\pi_3}\right).
$$
We proceed to compare \eqref{eq:eis} with ours.
The argument is similar to that for the R\'edei case.
The local computation for $[\chi_1,\chi_2,\chi_3]_3$ is given in Table\,\ref{table:amm}. Here the extension $F_{\mathfrak p_1,\mathfrak p_2}$ gives rise to $\epsilon \in C^2(G,\Z/3\Z)$ such that $\epsilon$ is unramified outside $\{\mathfrak p_1, \mathfrak p_2\}$ and that $d(\epsilon\cup\chi_3) = \chi_1 \cup \chi_2\cup\chi_3$.
Since there are no real places, we take $S=\{\mathfrak p_1,\mathfrak p_2, \mathfrak p_3\}$.
Since $\eta=\epsilon \cup \chi_3$ is unramified at $\sqrt{-3}$, there are no contributions at $\sqrt{-3}$, in view of Remark\,\ref{remark:2}.  
\\
\begin{table}[htb]
\def\arraystretch{1.2}
\begin{tabular}{|c||c|c|c|c|c|}
\hline
$v$		&	$\mathfrak p_1$	or  	$\mathfrak p_2$ 	& 	$\mathfrak p_3$	
\\
\hline 
$\operatorname{loc}_v^2\left(\eta^{\operatorname{glob}}\right)$
				& ramified 	&unramified 	
\\
\hline
$\chi_3$ 			&$0$					&ramified	
\\
\hline
$\eta_v$ 		&	$0$ 					& 0	
\\
\hline
$\operatorname{loc}_v\left(\epsilon\cup \chi_3\right) - \eta_v$
				&	$0$					&	$\psi_{\theta}\cup\psi_{\pi_3}$		
\\
\hline
\end{tabular}
\caption{local terms in the case $F=\mathbb Q(\zeta_3)$}
\label{table:amm}
\end{table}
\par
Then, except for the term at $\mathfrak p_3$, where we have $\inv_{\mathfrak p_3}(\psi_{\theta}\cup\psi_{\pi_3})$,
the local terms vanish.
In view of \eqref{eq:13}, we conclude that $[\mathfrak p_1,\mathfrak p_2,\mathfrak p_3]_3^{-1} = [\chi_1,\chi_2,\chi_3]_{3}$.
\end{proof}

\begin{corollary}
For a permulation $\sigma$ on $\{1,2,3\}$ with signe $|\sigma|$, we have
$$
[\mathfrak p_1,\mathfrak p_2, \mathfrak p_3]=
[\mathfrak p_{\sigma(1)},\mathfrak p_{\sigma(2)}, \mathfrak p_{\sigma(3)}]^{|\sigma|}.
$$
\end{corollary}
\begin{proof}
It is a consequence of our comparison and Cor.\,\ref{cor:alt}.
\end{proof}

\section{Comparison to the triple Massey products}\label{section:massey}

In this section, we show that our symbols $[\chi_1, \chi_2, \chi_3]_n$ are expressed in terms of the Massey products of the characters $\chi_1,\chi_2,\chi_3$
in the cases discussed in the sections 4 and 5.  We keep the same notations as in the previous sections.

Let $\chi_1, \chi_2$ and $\chi_3$ be the characters $G = {\rm Gal}(\overline{F}/F) \rightarrow \mathbb{Z}/n\mathbb{Z}$ in the sections 3 and 4. 
   Namely, when $F = \mathbb{Q}$ and $n=2$, they are the Kummer characters associated to distinct prime numbers $p_1, p_2$ and $p_3$ satisfying the conditions $p_i \equiv 1$ mod 4 ($i = 1,2,3$) and $\left( \frac{p_i}{p_j} \right) = 1$ ($1 \leq i < j \leq 3$), which are $(-1)^{\chi_i(g)} = g(\sqrt{p_i})/\sqrt{p_i}$ ($i=1,2,3$). When  when $F = \mathbb{Q}(\zeta_3)$ and $n=3$, they are the Kummer characters associated to distinct maximal ideals $\frak{p}_1 = (\pi_1), \frak{p}_2 = (\pi_2)$ and $\frak{p}_3 = (\pi_3)$ satisfying the conditions ${\rm N}\frak{p}_i \equiv 1$ mod 9 ($i = 1,2,3$) and $\left( \frac{\pi_i}{\pi_j} \right)_3 = 1$ ($1 \leq i < j \leq 3$), which are defined by  $\zeta_3^{\chi_i(g)} = g(\sqrt[3]{\pi_i})/\sqrt[3]{\pi_j}$ ($i=1,2,3$). 
   
As we see in the sections 3 and 4, we have
$$ {\rm loc}_v^2(\chi_i \cup \chi_j) = 0 \;\;\;\; (i\neq j) $$
for any place of $F$. By the global reciprocity, 
$$\bigoplus_v {\rm loc}_v^2 : H^2(G, \mathbb{Z}/n\mathbb{Z}) \rightarrow \bigoplus_v H^2(G_v, \mathbb{Z}/n\mathbb{Z}),$$
where $v$ runs over all places of $F$, is injective and hence
$$ \chi_i \cup \chi_j = 0 \;\;\;\; (i \neq j)$$
in $H^2(G, \mathbb{Z}/n\mathbb{Z})$. So there are $c_{12}, c_{23} \in C^1(G, \mathbb{Z}/n\mathbb{Z})$ such that
$$ \chi_1 \cup \chi_{23} = d c_{12}, \;\; \chi_2 \cup \chi_3 = d \chi_{23}.$$
The Massey product $\langle \chi_1, \chi_2, \chi_3 \rangle$ of $\chi_1, \chi_2, \chi_3$ is then defined by
$$ \langle  \chi_1, \chi_2, \chi_3 \rangle := [ c_{12} \cup \chi_3 + \chi_1 \cup c_{23}],$$
which is independent of choices of $c_{12}$ and $c_{23}$.                                              
\begin{theorem} Let $v$ denote $p_3$ $($resp. $\frak{p}_3)$ for the case that $n= 2, F =\mathbb{Q}$ $($resp. $n=3, F=\mathbb{Q}(\zeta_3))$, and set 
$\langle \chi_1, \chi_2, \chi_3 \rangle_v := {\rm inv}_{v}({\rm loc}_{v}^2(\langle \chi_1,\chi_2,\chi_3 \rangle))$ for simplicity.
Then we have
$$ \frak{d}(\chi_1,\chi_2,\chi_3) = \langle \chi_1, \chi_2, \chi_3 \rangle_v.$$
Consequently, we have
$$[p_1, p_2, p_3] = (-1)^{\langle \chi_1, \chi_2, \chi_3 \rangle_{p_3}}, \;\; [\frak{p}_1, \frak{p}_2, \frak{p}_3] = \zeta_3^{-\langle \chi_1, \chi_2, \chi_3 \rangle_{\frak{p}_3}}.$$
\end{theorem}
\begin{proof} Since ${\rm loc}_{v}^1(\chi_1) = 0$ by the assumptions, we have
$$ {\rm loc}_{v}^2(c_{12} \cup \chi_3 + \chi_1 \cup c_{23}) = {\rm loc}_{v}^1(c_{12} \cup \chi_3).$$
Observing that $c_{12}$ can be taken to be the cochain $\epsilon$ in the proofs of Theorem 3 and Theorem 4, we obtain the first assertion. The second assertions follow from Theorem 2 and Theorem 3.
\end{proof}
\begin{remark}\label{remark:3}
In [2] and  [8], the R\'{e}dei symbol and the triple cubic residue symbol were expressed by the Massey products, in a group-theoretic manner, using a (link group like) presentation of the maximal pro-$\ell$ quotient of $G_S$ and the transgression map, where $\ell = 2$ and $S = \{ p_1, p_2, p_3, \infty \}$ for the case of the R\'{e}dei symbol or $l=3$ and $S = \{ \frak{p}_1, \frak{p}_2, \frak{p}_3, \frak{p}_0 \}$  for the case of the triple cubic residue symbol, where $\frak{p}_0$ is a maximal ideal of $\mathbb{Z}[\zeta_3]$ such that ${\rm N}\frak{p}_0 \equiv 4$ or 7 mod 9. Theorem 6 gives a different cohomological  approach to the Massey product expression of the triple symbols.
\end{remark}
\bibliographystyle{plain}
\bibliography{ref}

\end{document}